\documentclass[11pt]{amsart}
\usepackage{amsmath,amsfonts,amssymb,amsthm}
\usepackage[alphabetic]{amsrefs}
\usepackage{hyperref}
\usepackage{graphicx,color}
\usepackage{pictexwd,dcpic,epsf}
\usepackage{enumerate}
\usepackage[OT2, T1]{fontenc}
\usepackage[russian,english]{babel}

\newtheorem{theorem}{Theorem}

\newtheorem{prop}[theorem]{Proposition}

\newtheorem{corollary}{Corollary}

\newtheorem*{WC2}{Whitehead's Problem}
\newtheorem*{atheorem}{Theorem A}

\theoremstyle{definition} 
\newtheorem{definition}{Definition}

\theoremstyle{remark}

\newcommand{\wt}{\widetilde}

\newcommand {\mr}{\mathrm}

\newcommand {\bZ}{\mathbb Z}

\newcommand{\aut}{\mathrm{Aut}}

\begin{document}

\title[
Normal subgroups of SimpHAtic groups
]{
Normal subgroups of SimpHAtic groups
}

\author{Damian Osajda}
\address{Instytut Matematyczny,
Uniwersytet Wroc\l awski\\
pl.\ Grun\-wal\-dzki 2/4,
50--384 Wroc{\l}aw, Poland}
\address{Institute of Mathematics, Polish Academy of Sciences\\
\'Sniadeckich 8, 00-656 War\-sza\-wa, Poland}
\email{dosaj@math.uni.wroc.pl}
\subjclass[2010]{{20F65, 20F67, 20F69, 57M07, 20F06}} \keywords{systolic group, asymptotic asphericity, normal subgroup, topology at infinity}

\thanks{Partially supported by Narodowe Centrum Nauki, grants no.\ UMO-2012/06/A/ST1/00259 and UMO-2017/25/B/ST1/01335. This work was partially supported by the grant 346300 for IMPAN from the Simons Foundation and the matching 2015-2019 Polish MNiSW fund.}
\date{\today}

\begin{abstract}
A group is SimpHAtic if it acts geometrically on a simply connected simplicially hereditarily aspherical (SimpHAtic) complex.
We show that finitely presented normal subgroups of the SimpHAtic groups
are either:
\begin{itemize}
\item
finite, or
\item
of finite index, or
\item
virtually free.
\end{itemize}
This result applies, in particular, to normal subgroups of systolic groups. We prove similar strong restrictions on group extensions for other classes of asymptotically aspherical groups.
 The proof relies on studying homotopy types at infinity of groups in question.
 
 We also show that nonuniform lattices in SimpHAtic complexes (and in more general complexes) are not finitely presentable and that finitely presented groups acting properly on such complexes act geometrically on SimpHAtic complexes. 

In Appendix we present the topological $2$--dimensional quasi-Helly property of systolic complexes.
\end{abstract}

\maketitle

\section{Introduction}
\label{s:intro}
\emph{Systolic complexes} were introduced for the first time by Chepoi \cite{Che1} under the name \emph{bridged complexes}. They were
defined
there as flag simplicial completions of \emph{bridged graphs} \cites{SoCh, FaJa} -- the latter appear naturally in many contexts 
related to metric properties of graphs and, in particular, convexity. Systolic complexes were rediscovered by Januszkiewicz-{\' S}wi{\c a}tkowski \cite{JS1} and Haglund \cite{Hag}, independently, in the context of geometric group theory. Namely, \emph{systolic groups}, that is, groups acting geometrically on systolic complexes have been used to 
construct important new examples of highly dimensional Gromov hyperbolic groups -- see e.g.\ \cites{JS0,JS1,O-chcg,OS}. 

Despite the fact that systolic groups and complexes can have arbitrarily high dimension\footnote{Here, we refer to various notions of \emph{dimension}: the (virtual) cohomological one, the asymptotic one, the geometric one...}, in many aspects they behave like two-dimensional objects.
This phenomenon was studied thoroughly in \cites{JS2,O-ci,O-ib,Sw,OS} (see also Appendix), and the intuitive picture is that systolic complexes and groups 
``asymptotically do not contain essential spheres''. 
Such behavior is typical for objects of dimension at most two, but systolic complexes and groups, and related classes (see e.g.\ \cites{ChOs,O-sdn,OS}) are the first highly dimensional counterparts. This implies, in particular, that systolic groups are very distinct from many classical groups, e.g.\ from uniform lattices in $\mathbb H^n$, for $n\geqslant 3$ -- see \cites{JS2,O-ci,O-ib,OS,Gom}.
\medskip

In the current article we present an algebraic counterpart of the asymptotic asphericity phenomenon mentioned above. Recall
that a theorem of Bieri \cite[Theorem B]{Bie} says that a finitely presented normal subgroup of a group of cohomological dimension at most two is either free or of finite index. Surprisingly, such strong restrictions on normal subgroups can hold also in the case of objects of  
arbitrarily high dimension, as the two following main results show. A group is \emph{SimpHAtic}\footnote{This term is meant rather to 
be associated with the words: \emph{sympatyczny} (Polish), \emph{sympathisch} (German), \emph{sympathique} (French), and \emph{\foreignlanguage{russian}{simpatichny\u i}} (Russian).} 
if it acts geometrically on a
 simply connected simplicially hereditarily aspherical (SimpHAtic) complex (see Subsection~\ref{s:SimpHAtic} for the precise definition).
\begin{theorem}
\label{t:0}
A finitely presented normal subgroup of a SimpHAtic group is either of finite index or virtually free.
\end{theorem}

The following corollary is an immediate application of Theorem~\ref{t:0} to specific classes of SimpHAtic groups.

\begin{corollary}	
\label{c:0}
Finitely presented normal subgroups of groups acting geometrically on either:
\begin{enumerate}
\item[(a)] systolic complexes, or
\item[(b)] weakly systolic complexes with $SD_2^{\ast}$ links, or
\item[(c)] simply connected nonpositively curved $2$--complexes, or
\item[(d)] simply connected graphical small cancellation complexes
\end{enumerate}
are of finite index or virtually free.
\end{corollary}

To the best of our knowledge, even in the case of small cancellation groups the result is new.\footnote{Note that not all small cancellation groups fulfill the assumptions of Bieri's theorem \cite{Bie}, and that (d) in Corollary~\ref{c:0} goes beyond the case of (graphical) small cancellation groups as understood usually.} 
However, the main point is that it holds also for higher dimensional groups. The case (a) establishes a conjecture by Dani Wise \cite[Conjecture 11.8]{Wise} (originally concerning only torsion-free groups) and answers some questions from \cite[Question 8.9(2)]{JS2}. A related conjecture \cite[Conjecture 11.6]{Wise} concerning description of centralizers of hyperbolic elements has been established in \cite{OsPr}.
The case (b) deals with a subclass of weakly systolic groups studied in \cites{O-sdn,ChOs,OS}, containing systolic groups and exhibiting the same asphericity phenomena. In particular, the corollary extends some results from \cite{JS2,O-ci,O-sdn,OS}.
For other classes of SimpHAtic groups to which Theorem~\ref{t:0} applies see e.g.\ \cite[Examples 4.4]{OS}.

Theorem~\ref{t:0} is a consequence of the following more general result. ($\pi_n^{\infty}$ denotes the $n^{\mr{th}}$ homotopy pro-group
at infinity -- see Subsection~\ref{s:prohom}.)

\begin{theorem}	
\label{t:2}
Let $K\rightarrowtail G \twoheadrightarrow Q$ be an exact sequence of infinite groups.
Assume that $G$ and $K$ have type $F_{\infty}$, $\pi_1^{\infty}(G)\neq 0$, and that either:
\begin{enumerate}
\item
$K$ is not virtually free and acts geometrically on a finite-dimensional contractible complex, or
\item
$Q$ is not virtually free and acts geometrically on a finite-dimensional contractible complex.
\end{enumerate}
Then 
$\pi_{n}^{\infty}(G)\neq 0$, for some $n\geqslant 2$.
\end{theorem}

This theorem, although quite technical in nature, is of its own interest and allows to conclude many other results
around normal subgroups of groups with some asphericity properties. 
We present below only three consequences of Theorem~\ref{t:2}.
(The notion of asymptotic hereditary asphericity is discussed in Subsection~\ref{s:AHA}.) 

\begin{corollary}
\label{c:1}	
Let $K$ be a finitely presented normal subgroup of an asymptotically hereditarily aspherical (AHA)
group $G$ of finite virtual cohomological dimension. 
Then $K$ is either of finite index or virtually free.
\end{corollary}

Theorem~\ref{t:2} implies also similar restrictions on quotients.

\begin{corollary}
\label{c:1a}	
Let $K\rightarrowtail G \twoheadrightarrow Q$ be an exact sequence of infinite groups with finitely presented
kernel $K$.
Assume that $G$ is a hyperbolic AHA group of finite virtual cohomological dimension. Then $K$ is virtually non-abelian free, and $Q$ is virtually free.
\end{corollary}

Questions concerning AHA group extensions appear e.g.\ in \cite[Question 8.9(2)]{JS2} and some partial results are 
presented in \cites{JS2, OS}. 
Note that all the new examples of highly dimensional hyperbolic AHA groups constructed
in \cites{JS0,JS1,O-chcg,OS} are of finite virtual cohomological dimension as in Corollary~\ref{c:1a}. 
It is not known in general whether hyperbolic systolic groups have finite virtual cohomological dimension.
Nevertheless, the following result applies to all of them.
\begin{corollary}
\label{c:3}
Let $K\rightarrowtail G \twoheadrightarrow Q$ be an exact sequence of infinite groups with finitely presented
kernel $K$. Assume that $G$ is a SimpHAtic hyperbolic group.
Then $K$ is virtually non-abelian free, and $Q$ is virtually free.
\end{corollary}

Finally, we present two results concerning groups acting properly on SimpHAtic complexes and on more general complexes.
They follow easily from known results, but do not appear explicitly in the literature.

Recall that for a locally compact topological group $G$ with left-invariant Haar measure $\mu$, a  \emph{lattice} 
is a discrete subgroup $H<G$ such that $\mu(H\backslash G)<\infty$.
A lattice $H<G$ is \emph{uniform} if $H\backslash G$
is compact. Lattices  -- e.g.\ arithmetic lattices -- form a rich source of examples of groups, and studying them is a classical and very active field of research, cf.\ \cite{Ser1971,BaLu2001,FaHruTho}.
One important class of lattices is obtained in the following way. Let $X$ be 
a locally finite, contractible, finite dimensional CW-complex.
We consider the group $G = \aut (X)$ of \emph{admissible} automorphisms of $X$ (set-wise stabilizer of each cell coincides
with its point-wise stabilizer). Note that any group of automorphisms of a CW-complex acts by admissible automorphisms on a subdivision of the complex.  
The group $G = \aut (X)$ naturally has
the structure of a locally compact topological group, where a neighborhood
basis of the identity consists of automorphisms of $X$ that are the
identity on bigger and bigger balls. 
Lattices in $G=\aut(X)$ are also referred to as \emph{lattices in $X$}.

Uniform lattices in $X$ are clearly finitely presented and, in fact, of type F$_\infty$. 
The question of finiteness conditions for nonuniform lattices in $X$ is, in general, important and difficult, cf.\ \cite{BaLu2001,FaHruTho,Gan2012}.
For example, it is a classical result that nonuniform lattices in locally finite trees are not finitely generated, see \cite{BaLu2001}. There exist nonuniform lattices in $2$-dimensional CAT$(0)$ complexes with Kazhdan's property
(T), hence finitely generated, see \cite{CaMaSteZa1993,Zuk1996,BaSw1997,DyJa2002}. (Note that Kazhdan introduced his property (T) in order
to prove finite generation of lattices in higher rank Lie groups.) Farb, Hruska, and Thomas \cite[Conjecture 35]{FaHruTho} conjectured 
that nonuniform lattices in locally finite contractible $2$-complexes are not finitely presentable. 
Gandini proved it (and a more general fact concerning finiteness conditions) in an elegant
way in \cite{Gan2012}.  
The following theorem extends Gandini's result to a large class of complexes of arbitrarily high dimension, in particular to many high dimensional (hyperbolic) buildings.

\begin{theorem}
	\label{t:1a}
	For a finite dimensional locally finite AHA complex $X$, non\-uniform lattices in $G=\aut(X)$ are not finitely presentable.
\end{theorem}

Let $\mathcal{C}$ be a class of locally finite combinatorial complexes (that is, CW-complexes with attaching maps restricted to open cells being homeomorphisms) closed  
under taking full subcomplexes and covers (see Section~\ref{s:t1t2} for details). The following result is implicit in \cite{HaMP} (and in e.g.\ \cite{Wise,Zad2014} for some subclasses), but we did not find an explicit formulation. It might be seen as a strengthening of Theorem~\ref{t:1a} in particular cases. 

\begin{theorem}
	\label{t:2a}
	Let a finitely presented group $G$ act properly on a simply connected complex from $\mathcal{C}$. Then $G$ acts geometrically on a simply connected complex from $\mathcal{C}$.
\end{theorem}

\begin{corollary}
	\label{c:fss0}
	A finitely presented group acting properly on a locally finite simply connected SimpHAtic complex is SimpHAtic.
\end{corollary}

\begin{corollary}
	\label{c:fss}
	For $k\geqslant 6$, a finitely presented group acting properly on a locally finite $k$--systolic complex is $k$--systolic.
\end{corollary}

\subsection*{Ideas of the proofs}
\label{s:idea}
The proof of Theorem~\ref{t:2} relies on the analysis of the homotopy type at infinity (proper homotopy type) of the group $G$.
Since $K,Q$ are infinite, $G$ has one end, and furthermore $K,Q$ are not one-ended, by $\pi_1^{\infty}(G)\neq 0$ (see Section~\ref{s:pfB}). Therefore,
by theorems of Stallings and of Dunwoody the homotopy type at infinity of $K,Q$ is the type of a disjoint union of 
some points and some homotopically nontrivial pieces (see Subsections~\ref{s:1red} \& \ref{s:kpi1s}). 
Furthermore, the homotopy type at infinity of $G$ is the join of the types of
$K$ and $Q$ (see Subsections~\ref{s:YZMN} \& \ref{s:vircyc}). However, such join contains a nontrivial $n$--sphere appearing in suspensions of some factors 
coming from $K$, for some $n\geqslant 2$. This happens because $K\times \bZ$ acts geometrically on a finite-dimensional
contractible complex, and thus $\pi_n^{\infty}(K_1 \times \bZ)\neq 0$, for a one-ended factor $K_1$
of $K$, and some $n\geqslant 2$.
Therefore, $\pi_n^{\infty}(G)\neq 0$.

For Theorem~\ref{t:0}, one observes that $\pi_n^{\infty}(G)=0$, for all $n\geqslant 2$, and that finitely presented subgroups
of SimpHAtic groups are again SimpHAtic.

\subsection*{Article's structure}
\label{s:str}
In the next Section~\ref{s:prel} we present some preliminary notions and results concerning e.g.: weakly systolic complexes and groups (Subsection~\ref{s:sys}), proper homotopy type (Subsection~\ref{s:prohom}), asymptotic hereditary asphericity (Subsection~\ref{s:AHA}), and SimpHAtic groups (Subsection~\ref{s:SimpHAtic}). 
In Section~\ref{s:pfB} we prove the main Theorem~\ref{t:2}. In Section~\ref{s:pfA} we present the proofs of Theorem~\ref{t:0}
 and Corollary~\ref{c:0}. 
In Section~\ref{s:fin} we prove Corollaries~\ref{c:1}, \ref{c:1a} \& \ref{c:3}.
In Section~\ref{s:t1t2} we prove Theorems~\ref{t:1a}\&~\ref{t:2a}.

In Appendix we present the topological $2$--dimensional quasi-Helly property of systolic complexes -- Theorem A. This is yet another appearance of an asymptotic asphericity in the systolic framework.  
\medskip

\noindent
{\bf Acknowledgment.} I thank the anonymous referee for pointing out that the more general version of Theorem~\ref{t:2} appearing initially in the paper is incorrect. I thank all the referees for other corrections and many valuable remarks
improving the exposition.

\section{Preliminaries}
\label{s:prel}

\subsection{Complexes and groups}
\label{s:cplx}
Throughout this article we work usually (if not stated otherwise) with regular CW complexes, called further simply \emph{complexes} (see e.g.\ \cite{Geo-book} for some notation). 
Accordingly, all \emph{maps} are cellular maps.
We call a complex \emph{$G$--complex} if a group $G$ acts on it by cellular automorphisms. 
The action is \emph{proper}, or the complex is a \emph{proper $G$--complex} if stabilizers of cells are finite. 
The $n$--skeleton of a complex $X$ is denoted by $X^{(n)}$. 
A $G$--complex $X$ has the $n$--skeleton \emph{finite mod $G$} if the quotient $X^{(n)}/G$ is a finite complex.
Recall that a group has \emph{type $F_n$} if there is a \emph{free} (that is, with trivial stabilizers of cells) contractible $G$--complex with the $n$--skeleton finite mod $G$.
A group $G$ acts \emph{geometrically} on $X$, if $X$ is a finite-dimensional proper contractible $G$--complex with every skeleton finite mod $G$. 
In particular, $X$ is then uniformly locally finite.
We usually work with the path metric on $X^{(0)}$ defined by lengths of paths in $X^{(1)}$ connecting two given vertices.
It follows that if a finitely generated $G$ acts geometrically on $X$, then $X^{(0)}$ is quasi-isometric to $G$ equipped with a word metric coming from a finite generating set. For a subcomplex $Y$ of a complex $X$, by $X - Y$ we denote the smallest subcomplex of 
$X$ containing $X\setminus Y$ (set-theoretic difference).

\subsection{Weakly systolic complexes and groups}
\label{s:sys}
In this  subsection we restrict our studies to simplicial complexes. We follow mostly notation from e.g.\ \cites{JS1,O-ci,O-ib,O-sdn,O-chcg}.
A simplicial complex $X$ is \emph{flag} if it is determined by its $1$--skeleton $X^{(1)}$, that is, if every set of vertices pairwise connected
by edges spans a simplex in $X$. A subcomplex  $Y$ of a simplicial complex $X$ is \emph{full} if it is determined by its vertices, that is,
if any set of vertices in $Y$ contained in a simplex of $X$, spans a simplex in $Y$. 
The \emph{link} of a simplex $\sigma$ of $X$ is the simplicial complex being the union of all simplices $\tau$ of $X$ disjoint from $\sigma$, but spanning together with $\sigma$ a simplex of $X$.  A $k$--cycle is a triangulation of the circle $S^1$ consisting of $k$ edges and $k$ vertices. For $k\geqslant 5$, a flag simplicial complex $X$
is \emph{$k$--large} (respectively, \emph{locally $k$--large}) if there are no full $i$--cycles (as subcomplexes) in $X$ (respectively, in any link of $X$), for $i<k$. A complex $X$ is \emph{$k$--systolic} (respectively, \emph{systolic}) if it is simply connected and locally $k$--large (respectively, locally $6$--large) -- see \cite{JS1}.  
A \emph{$k$--wheel} is a graph consisting of a $k$--cycle and a vertex adjacent to all vertices of the cycle. A \emph{$k$--wheel with
a pendant triangle} is a graph consisting of a $k$--wheel and a vertex adjacent to two adjacent vertices of the $k$--cycle.
A complex $X$ satisfies the \emph{$SD_2^{\ast}$ property} if it is locally $5$--large and every $5$--wheel with a pendant triangle is
contained in a $1$--ball around some vertex of $X$. A flag simplicial complex is \emph{weakly systolic} if it is simply connected and satisfies
the $SD_2^{\ast}$ property -- see \cite{O-sdn,ChOs}. A weakly systolic complex in which all links satisfy the $SD_2^{\ast}$ property is
called a \emph{weakly systolic complex with $SD_2^{\ast}$ links}. A group is \emph{systolic} (respectively, \emph{$k$--systolic}, 
\emph{weakly systolic})
if it acts geometrically on a systolic (respectively, $k$--systolic, weakly systolic) complex.

\subsection{Proper homotopy type}
\label{s:prohom}
The basics about proper homotopy theory can be found in \cite{Geo-book}. 
Recall that a cellular map $f\colon X\to Y$ between complexes 
is \emph{proper} if $f^{-1}(C)\cap X^{(n)}$ is a finite subcomplex for any finite subcomplex $C\subset Y$ and any $n$ (see e.g.\ \cite[Chapter 10]{Geo-book}, where the name \emph{CW-proper} is used).
Two maps $f,f'\colon X\to Y$ are \emph{properly homotopic}, denoted $f{\simeq_p} f'$, if there exists a proper map $F\colon X\times [0,1] \to Y$ with 
$F(x,0)=f(x)$ and $F(x,1)=f'(x)$, for all $x\in X$. A proper map $f\colon X \to Y$ is a \emph{proper homotopy equivalence} (respectively, \emph{proper $n$--equivalence})
if there exists a \emph{proper homotopy inverse} (respectively, \emph{proper $n$--inverse}) $f^{-1}\colon Y \to X$ with $f^{-1}\circ f \simeq_p \mr{id}_X$, and
$f\circ f^{-1} \simeq_p \mr{id}_Y$
(respectively, $f^{-1}\circ f|_{X^{(n)}} \simeq_p \mr{id}_{X^{(n)}}$, and
$f\circ f^{-1}|_{Y^{(n)}} \simeq_p \mr{id}_{Y^{(n)}}$). In such case we say that $X$ and $Y$ have the same \emph{proper homotopy type} (respectively, \emph{proper $n$--type}).

For $n\geqslant 0$, we say that the \emph{$n^{th}$--homotopy pro-group at infinity vanishes} for $X$, denoted $\pi_n^{\infty}(X)=0$, if the following
holds. For every finite subcomplex $M\subset X$ there exists a finite subcomplex $M'$ containing $M$ such that: For any map $f\colon S^n \to
X - M'$ of the $n$--sphere $S^n$, there exists its extension $F\colon B^{n+1} \to X - M$. 
If $\pi_n^{\infty}(X)=0$ does not hold, then we write $\pi_n^{\infty}(X)\neq 0$. The condition $\pi_0^{\infty}(X)=0$ means that $X$ is \emph{one-ended}, and $\pi_0^{\infty}(X)=\pi_1^{\infty}(X)=0$ means \emph{simple connectedness at infinity}.
The next proposition shows that for a group acting in a nice way on a complex, vanishing of homotopy pro-groups at infinity
does not depend on the complex. It is a version of \cite[Proposition 2.10]{O-ci} whose proof is analogous to the one in \cite{O-ci}. (We may 
omit the requirement of `rigidity' here, since an action induced on the subdivision of a regular CW complex is rigid.)
Compare also
\cite[Theorems 17.2.3 \& 17.2.4]{Geo-book} and \cite[Theorem 7.3]{OS}, and proofs there. 

\begin{prop}
\label{p:conninf}
Let $X,Y$ be two $n$--connected proper $G$--complexes with $(n+1)$--skeleta finite mod $G$. Then $\pi_n^{\infty}(X)=0$ 
iff $\pi_n^{\infty}(Y)=0$.
\end{prop}

Therefore, we say that $\pi_n^{\infty}(G)=0$ (respectively, $\pi_n^{\infty}(G)\neq 0$) if $\pi_n^{\infty}(X)=0$ (respectively, $\pi_n^{\infty}(X)\neq 0$), for some (and hence, any) $G$--complex as in Proposition~\ref{p:conninf}. 

Similarly (cf.\ e.g.\ \cite[Theorem 7.3]{OS}), vanishing of the homotopy pro-groups at infinity is an invariant of the coarse equivalence, and hence of quasi-isometry. However,
directly from the definition above it follows that this is also an invariant of the proper homotopy type. More precisely, if there
is a proper $(n+1)$--equivalence $f\colon X \to Y$, then $\pi_n^{\infty}(X)=0$ iff $\pi_n^{\infty}(Y)=0$ -- cf.\ \cite[Chapter 17]{Geo-book}.
This invariance is crucial in our approach -- see Subsection~\ref{s:kpi1s}.

\subsection{Group homology}
\label{s:homo}

In this subsection we present a strengthening of \cite[Proposition 2.9]{O-ci} concerning the cohomology of groups. 
The standard reference for the group homology is the book \cite{Bro} (see also \cites{Geo-book,Hat}).

\begin{prop}
\label{p:homo}
Let $G$ be a group acting cocompactly and with finite stabilizers on a contractible finite-dimensional complex $X$.
If $G$ is not virtually free then there exists $2\leqslant n \leqslant \mr{dim}(X)$ such that $H^n(G;\mathbb ZG)\neq 0$.
\end{prop}

\begin{proof}
As in the proof of \cite[Proposition 2.8]{O-ci} it follows that there exists a free $\mathbb ZG$--module $\widetilde F$ with
$H^n(G;\widetilde F)\neq 0$ or $H^{n+1}(G;\widetilde F)\neq 0$, where $n:=\mr{cd}_{\mathbb Q}G\leqslant \mr{dim}(X)$. As in the proof of
\cite[Proposition 2.9]{O-ci} it follows then that $H^n(G;\mathbb ZG)\neq 0$. Since $G$ is not virtually free, by \cite[Corollary 1.2]{Dun0}
we obtain that $n=\mr{cd}_{\mathbb Q}G\geqslant 2$.
\end{proof}

\subsection{Asymptotic hereditary asphericity (AHA)}
\label{s:AHA} 

The notion of asymptotic hereditary asphericity (AHA) was introduced in \cite{JS2}. We follow here mostly notations from \cite{OS}.
Given an integer $i\geqslant 0$, for subsets $C\subseteq D$ of a metric space $(X,d)$
we say that $C$ \emph{is $(n;r,R)$--aspherical in $D$} if every simplicial map $f\colon S\to P_r(C)$
(where $P_r$ denotes the Rips complex with constant $r$),
where $S$ is a triangulation of the $n$--sphere $S^n$, has a simplicial extension
$F\colon B \to P_R(D)$, for some triangulation $B$ of the $(n+1)$--ball $B^{n+1}$ such that $\partial B=S$.
A metric space $X$ is \emph{asymptotically hereditarily aspherical}, shortly \emph{AHA},
if for every $r>0$ there exists $R>0$ such that every subset $A\subseteq X$ is $(n;r,R)$--aspherical in itself,
for every $n\geqslant 2$. A finitely generated group is AHA if it is AHA as a metric space for some (and hence any) word
metric coming from a finite generating set. A subgroup of an AHA group is AHA \cite[Corollary 3.4]{JS2}.
A finitely presented AHA group $G$ has type $F_{\infty}$ \cite[Theorem C]{OS}, and $\pi_n^{\infty}(G)=0$ for all $n\geqslant 2$ 
\cite[Theorem D]{OS}. If moreover, the virtual cohomological dimension of $G$ is finite, or if $G$ acts geometrically on a contractible
complex of finite dimension, then $\pi_1^{\infty}(G)\neq 0$ \cite[Theorem 7.7]{OS}.

For examples of AHA groups see the next subsection.

\subsection{SimpHAtic groups}
\label{s:SimpHAtic}
In this subsection we introduce the notions of SimpHAtic complexes and SimpHAtic groups. 
\begin{definition}
A flag simplicial complex is \emph{simplicially hereditarily aspherical}, shortly \emph{SimpHAtic}, if every full subcomplex is
aspherical. 

A group acting geometrically on a simply connected SimpHAtic complex is called \emph{SimpHAtic}.
\end{definition}
\noindent
Let us note few simple facts about SimpHAtic groups:
\begin{enumerate}
\item
A simply connected SimpHAtic complex is contractible, by Whitehead's theorem (see e.g.\ \cite[Theorem 4.5]{Hat} or \cite[Proposition 4.1.4]{Geo-book}).
\item
SimpHAtic complexes and groups are AHA, by \cite[Corollary 3.6]{OS}. It follows, that for a one-ended SimpHAtic group $G$, $\pi_1^{\infty}(G)\neq 0$ \cite[Theorem 7.7]{OS}, and $\pi_n^{\infty}(G)=0$ \cite[Corollary 7.6]{OS}, for $n\geqslant 2$.

\item
Weakly systolic complexes with $SD_2^{\ast}$ links are SimpHAtic by \cite[Proposition 8.2]{O-sdn}. In particular, systolic complexes
are SimpHAtic, and systolic groups are SimpHAtic (see \cite{JS1}).
\item
There are SimpHAtic groups that are not weakly systolic -- e.g.\ some Baumslag-Solitar groups. For other examples of SimpHAtic
groups see e.g.\ \cite{OS}.
\item
Full subcomplexes and covers of SimpHAtic complexes are again SimpHAtic. Therefore, by a theorem of Hanlon-Mart{\'{\i}}nez-Pedroza
 \cite[Theorem 1.1]{HaMP}, finitely presented subgroups of SimpHAtic groups are again SimpHAtic.
\end{enumerate}

We cannot resist to formulate here the famous Whitehead's asphericity question in the following way:
\begin{WC2}
Are all aspherical $2$--complexes SimpHAtic?
\end{WC2}

\section{Proof of Theorem~\ref{t:2}}
\label{s:pfB}
In this section we present the proof of Theorem~\ref{t:2} from the Introduction.
Recall that $K\rightarrowtail G \twoheadrightarrow Q$ is an exact sequence of infinite groups. 
Since $K,G$ have type $F_{\infty}$, the image $Q$ has type $F_{\infty}$ as well -- see e.g.\ \cite[Theorem 7.2.21]{Geo-book}. 
Since $K,Q$ are infinite, we have that $G$ has one end, that is, $\pi_0^{\infty}(G)=0$ -- see e.g.\ \cite[Corollary 16.8.5]{Geo-book}.
By the result of Houghton \cite{Hou} and Jackson \cite{Jac} (see also \cite[Corollary 16.8.5]{Geo-book}), from the fact that $\pi_1^{\infty}(G)\neq 0$ we conclude that neither $K$ nor $Q$ has
one end, that is $\pi_0^{\infty}(K)\neq 0\neq \pi_0^{\infty}(Q)$. 
For the rest of the proof we focus on the case (1), that is we assume that $K$ is not virtually free and acts geometrically on a 
finite-dimensional contractible complex.
The proof in the other case -- under the assumptions on $Q$ -- is essentially the same (just exchanging the roles of $K$ and $Q$). 

\subsection{Decomposing $K$ and $Q$}
\label{s:1red}
Since the kernel $K$ has more than one end, by the result of Stallings \cite{Sta}, it decomposes into a free product with amalgamation
over a finite group, or as an HNN extension over a finite group. If the factors are not one-ended we may repeat decomposing further.
A theorem of Dunwoody \cite{Dun} says that finitely presented groups are accessible, that is, such a procedure ends after finitely many steps. As a result, by the classical Bass-Serre theory (see e.g.\ \cite[Chapter 6.2]{Geo-book}), the group $K$ is the fundamental group of a finite graph of groups $\mathcal K$ with vertex groups having at most one end, and finite edge groups. Let $\{ K_i \}$ be the collection of the vertex groups, and let $\{ L_j \}$ be the set of edge groups of $\mathcal K$.
We may assume that $K_i$ are nontrivial.
Since $K$ is not virtually free, by Proposition~\ref{p:homo}, we have that $H^{n'}(K;\mathbb ZK)\neq 0$ for some $n'\geqslant 2$.
Therefore, by \cite[Theorem 13.3.3]{Geo-book}, it follows that among the groups  $\{ K_i \}$, there is a one-ended one, say $K_1$, 
with $H^{n'}(K_1;\mathbb ZK_1)\neq 0$. By \cite[Theorem 17.3.5]{Geo-book} we have that $H^{n'+1}(K_1\times \mathbb Z;\mathbb Z(K_1\times \mathbb Z))\neq 0$.
By the aforementioned result of Houghton and Jackson, $\pi_1^{\infty}(K_1\times \mathbb Z)= 0=\pi_0^{\infty}(K_1\times \mathbb Z)$.
Therefore, by the Proper Hurewicz Theorem (cf.\ e.g.\ \cite[Theorem 17.1.6]{Geo-book}), it follows that there exists  $n\geqslant 2$ such that $\pi_n^{\infty}(K_1 \times \mathbb Z)\neq 0$ (compare the proof of \cite[Theorem 3.2]{O-ci}).

\subsection{Classifying spaces for $K$ and $Q$}
\label{s:kpi1s}
Since $K$ has type $F_{\infty}$, all the factors $K_i$ have also type $F_{\infty}$.
For all $i$, let $Y_i'$ be a contractible free $K_i$--complex with every skeleton finite mod $K_i$. By a general technique (see e.g.\ \cite[Chapter 6.2]{Geo-book}, and \cite[Chapter 1.B]{Hat}) one may combine copies of $Y_i'$ to obtain a tree of complexes $Y'$ being a contractible proper $K$--complex with each skeleton finite mod $K$. Furthermore, by modifying slightly each $Y_i'$ resulting in new complexes $Y_i$ (by e.g.\ coning-off $L_j$--invariant subcomplexes), we obtain a contractible proper $K$--complex $Y$ with every skeleton finite mod $K$ and satisfying the following properties:
\begin{enumerate}
\item
$Y$ is a union of copies of $Y_i$;
\item
any two such copies intersect in at most one vertex, and every vertex belongs to at most two such copies;  
\item
every $Y_i$ is a contractible proper $K_i$--complex with every skeleton finite mod $K_i$.
\end{enumerate}
Observe that, by our choice, the complex $Y_1$ has one end.

Similarly, we construct a contractible proper $Q$--complex $Z$ with each skeleton finite mod $Q$. We distinguish two cases, which will be treated separately in the two following subsections:

\medskip
\noindent
\emph{Case A: $Q$ has infinitely many ends.} In this case the complex $Z$ may be chosen to
be a union of infinite contractible complexes $\{ Z_i\}$ (we do not require that they correspond to 
the Bass-Serre decomposition of $Q$), such that any two such complexes intersect in at most one vertex, and every vertex belongs to at most two such complexes. We require here that there are at least two such complexes $Z_1,Z_2$. For example, if $Q$ is virtually free then $Z$ can
be a tree consisting of $Z_i$ being lines. If $Q$ is not virtually free then $Z_1$ may be taken to
be a complex corresponding to a one-ended factor of $Q$ (but we do not require further that $Z_1$ is one-ended).    

\medskip
\noindent
\emph{Case B: $Q$ has two ends.} In this case $Q$ is virtually cyclic, and the complex $Z$ is taken to be a complex homeomorphic to the line $\mathbb R$.
\medskip

Since $Y$ and $Z$ have skeleta finite mod $K$ and $Q$, and since $G$ has type $F_{\infty}$, there is a contractible proper $G$--complex $X$ with the following properties (see e.g.\ \cite[Proposition 17.3.4]{Geo-book}): Each skeleton of $X$ is finite mod $G$, and there is a proper homotopy equivalence between $X$ and $Y\times Z$. 
Therefore, to prove the theorem, it is enough to show that $\pi_n^{\infty}(Y\times Z)\neq 0$: By the proper $(n+1)$--equivalence, it implies $\pi_n^{\infty}(X)\neq 0$ and 
hence, by Proposition~\ref{p:conninf}, $\pi_n^{\infty}(G)\neq 0$. 

In the following two subsections 
we prove that $\pi_n^{\infty}(Y\times Z)\neq 0$ separately for two Cases: A and B. 
%
To do this we need to find subcomplexes $M\subset Y$ and $N\subset Z$ with finite $(n+1)$--skeleta, and having the following property: For any $C\subset Y\times Z$ with finite $(n+1)$--skeleton and containing $M\times N$, there is a map $S^n \to Y\times Z - C$ that is not homotopically trivial within $Y\times Z - M\times N$. 

\subsection{Case A: $Q$ has infinitely many ends.}
\label{s:YZMN}

Pick a copy of $Y_1$ in $Y$ (denoted further simply $Y_1$), and a subcomplex $Z_1$ of $Z$. 
Since $\pi_{n}^{\infty}(K_1\times \mathbb Z)\neq 0$, by Proposition~\ref{p:conninf}, $\pi_{n}^{\infty}(Y_1 \times \mathbb R)\neq 0$ -- the space $Y_1 \times \mathbb R$ (with the product structure of a complex) is a proper $(K_1 \times \mathbb Z)$--complex
with skeleta finite mod $K_1\times \mathbb Z$. Therefore, there exists a subcomplex $M\subset Y_1$ with finite $(n+1)$--skeleton and such that the following holds: 
There is a compact interval $I=[c,d]\subset \mathbb R$ such that for every $M'\subset Y_1$ with finite $(n+1)$--skeleton and containing $M$, and 
for every compact interval $I'\subset \mathbb R$ containing $I$, there exists a singular sphere $S^{n} \to Y_1 \times \mathbb R - M' \times I'$ not homotopically 
trivial within $Y_1 \times \mathbb R - M\times I$. 
Furthermore, we choose a subcomplex $N$ of $Z_1$ with finite $(n+1)$--skeleton, containing the vertex $b_1$ being the intersection of $Z_1$ and $Z_2$, and all cells in $Z_1$ containing $b_1$. This implies, in particular,
that $b_1\notin Z_1 - N$ but $b_1\in Z - N$ (because there are cells in $Z_2$ containing properly $b_1$).

\medskip
We claim that $M,N$ are as required above. To show this, let $C\subset Y \times Z$ 
be a subcomplex with finite $(n+1)$--skeleton and containing $M\times N$.
In the rest of the subsection we construct a singular sphere $\varphi \colon S^n \to Y \times Z - C$ not homotopically trivial within $Y \times Z  - M\times N$.  
\medskip

Choose $M'' \subset Y$ and $N'' \subset Z$ with finite $(n+1)$--skeleta such that $C\subset M'' \times N''$.
Let $M' = M''\cap Y_1$. 
Choose a simplicial path $l\colon I'=[c',d'] \to Z_1\cup Z_2$ (with the interval $[c',d']$ equipped with a simplicial
structure with integral vertices) with the following properties:
\begin{itemize}
	\item $c'<c$ and $d<d'$, that is, $I'=[c',d']$ contains the interval $I=[c,d]$ defined above;
	\item $l(c)=b_1$, and $l(c'),l(d')\in Z-N''$;
	\item $l([c',c])\subset Z_2$, $l([c,d])\subset Z_1$, $l([d,d'])\subset Z_1-N$. 
\end{itemize}
Let $f\colon Y_1\times \mathbb{R} - M\times I\to Y\times Z$ be a cellular map defined as follows:
\begin{itemize}
	\item $f(y,t)=(y,l(c'))$, for $t<c'$;
	\item $f(y,t)=(y,l(t))$, for $t\in[c',d']$;
	\item $f(y,t)=(y,l(d'))$, for $t>d'$.
\end{itemize}
Observe that the image of $f$ is contained in $Y\times Z - M \times N$, and that
$f(Y_1\times \mathbb R - M' \times I')\subset Y\times Z - M''\times N''$. 
Let $\overline{\varphi}\colon S^n \to Y_1\times \mathbb R - M' \times I'$ be a singular sphere
not homotopically 
trivial within $Y_1 \times \mathbb R - M\times I$. 
In the rest of the subsection we show that the singular sphere $\varphi = f \circ 
\overline{\varphi}\colon S^n \to Y\times Z - C$ is not homotopically trivial within
$Y\times Z - M \times N$. This requires a detailed analysis of the topology of the latter space.
\medskip
%

The $(n+1)$--skeleton of $Y$ (respectively, $Z$) is locally finite. Therefore, by our construction of $Y$ (respectively, $Z$), 
the set $A:=Y_1 \cap (Y - Y_1)$ (respectively, $B:=Z_1 \cap (Z - Z_1)$) is a discrete set of vertices.
Since the $(n+1)$--skeleton of $M$ (respectively, $N$) is finite, the set $A\cap M$ (respectively, $B\cap N$) is an empty or finite set $\{ a_1,a_2,\ldots, a_m \}$ (respectively, finite set $\{ b_1,b_2,\ldots, b_s \}$).
For $a\in A$ (respectively, $b\in B$) let $A_a$ (respectively, $B_b$) denote the closure in $Y$ (respectively, in $Z$) of the connected
component of $Y\setminus \{ a \}$ (respectively, $Z\setminus \{ b \}$), not containing $Y_1$ (respectively, $Z_1$). Observe that such closure
is the union of the component and $\{ a\}$ (respectively, $\{ b\}$) -- see Figure~\ref{f:EG0} (left).

\begin{figure}[h!]
\centering
\includegraphics[width=1\textwidth]{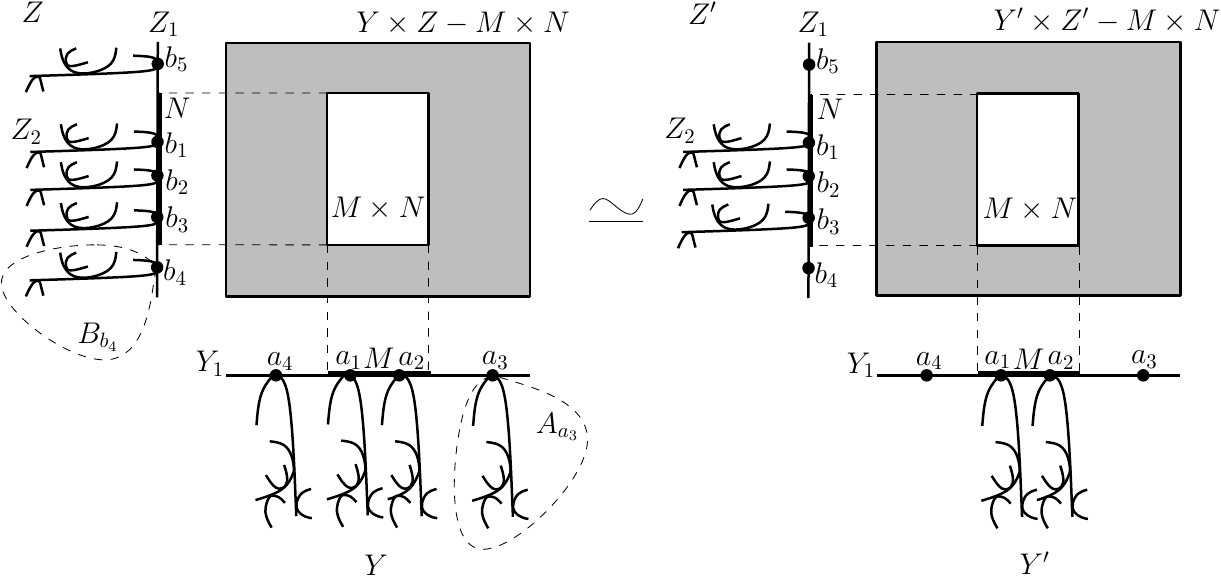}
\caption{From $Y\times Z - M \times N$ to $Y'\times Z' - M \times N$.}
\label{f:EG0}
\end{figure}

Since $A_a$ (respectively, $B_b$) is contractible, for $a\notin M$ (respectively, $b\notin N$) we may contract $A_a$ to $a$ (respectively, $B_b$ to $b$)
not changing the homotopy type of $Y\times Z - M\times N$. Therefore, let $Y'$ (respectively, $Z'$) denote $Y$ (respectively, $Z$) with 
$A_a$ (respectively, $B_b$) contracted to $a$ (respectively $b$), for all $a\notin M$ (respectively, $b\notin N$) -- see Figure~\ref{f:EG0} (right).
We have:
\begin{align}
\label{e:10}
\begin{split}
&Y'=Y_1\cup (A_{a_1}\sqcup A_{a_2}\sqcup \ldots \sqcup A_{a_m})\\
&Z'=Z_1\cup (B_{b_1}\sqcup B_{b_2}\sqcup \ldots \sqcup B_{b_s}),
\end{split}
\end{align}
and a homotopy equivalence
\begin{align}
\label{e:11}
h_1 \colon Y\times Z - M\times N \to Y' \times Z' - M\times N.
\end{align}
Furthermore, by formula (\ref{e:10}), and since $M\subset Y_1$ (respectively, $N\subset Z_1$) and $\{a_1,a_2,\ldots,a_m\}\subset M$ 
(respectively, $\{b_1,b_2,\ldots,b_s\}\subset N$), we have
\begin{align}
\label{e:12}
\begin{split}
Y' \times Z' - M\times N=& (Y_1\times Z_1 - M\times N) \cup \\ 
&\cup Y_1 \times (B_{b_1}\sqcup  \ldots \sqcup B_{b_s}) \cup \\
&\cup (A_{a_1}\sqcup  \ldots \sqcup A_{a_m}) \times Z_1 \cup \\
&\cup (A_{a_1}\sqcup  \ldots \sqcup A_{a_m}) \times (B_{b_1}\sqcup  \ldots \sqcup B_{b_s}),
\end{split}
\end{align}
where $Y_1\times Z_1 - M\times N$ and $Y_1 \times (B_{b_1}\sqcup  \ldots \sqcup B_{b_s})$ intersect along
$(Y_1 - M)\times (\{ b_1 \} \sqcup  \ldots \sqcup \{ b_s\} )$, etc.

Again, we may contract subcomplexes of the type $Y_1\times B_{b_j}$ to $Y_1\times \{b_j\}$, $A_{a_i}\times Z_1$ to $\{ a_i\} \times Z_1$, and $A_{a_i}\times B_{b_j}$ to $\{ a_i\} \times \{ b_j\}$,
obtaining the following homotopy equivalence -- see Figure~\ref{f:EG1}. 

\begin{align}
\label{e:13}
\begin{split}
h_2 \colon Y \times Z - M\times N\to & (Y_1\times Z_1 - M\times N) \cup \\ 
&\cup Y_1 \times (\{b_1\}\sqcup  \ldots \sqcup \{b_s\}) \cup \\
&\cup (\{a_1\}\sqcup  \ldots \sqcup \{a_m\}) \times Z_1 \cup \\
&\cup (\{a_1\}\sqcup  \ldots \sqcup \{a_m\}) \times (\{b_1\}\sqcup  \ldots \sqcup \{b_s\}),
\end{split}
\end{align}
where (see Figure~\ref{f:EG1}):
\begin{align}
\label{e:14}
\begin{split}
(Y_1\times Z_1 - M\times N) \cap [Y_1 \times (\{b_1\}\sqcup  \ldots \sqcup \{b_s\})]=\\
=(Y_1 - M)\times (\{b_1\}\sqcup  \ldots \sqcup \{b_s\}),\\
(Y_1\times Z_1 - M\times N) \cap [(\{a_1\}\sqcup  \ldots \sqcup \{a_m\})\times Z_1]=\\
=(\{a_1\}\sqcup  \ldots \sqcup \{a_m\})\times (Z_1 - N),\\
\mr{and}\: (\{a_i\} \times Z_1) \cap (Y_1 \times \{ b_j \}) = \{ a_i \} \times \{b_j \}.
\end{split}
\end{align}

\begin{figure}[h!]
\centering
\includegraphics[width=1\textwidth]{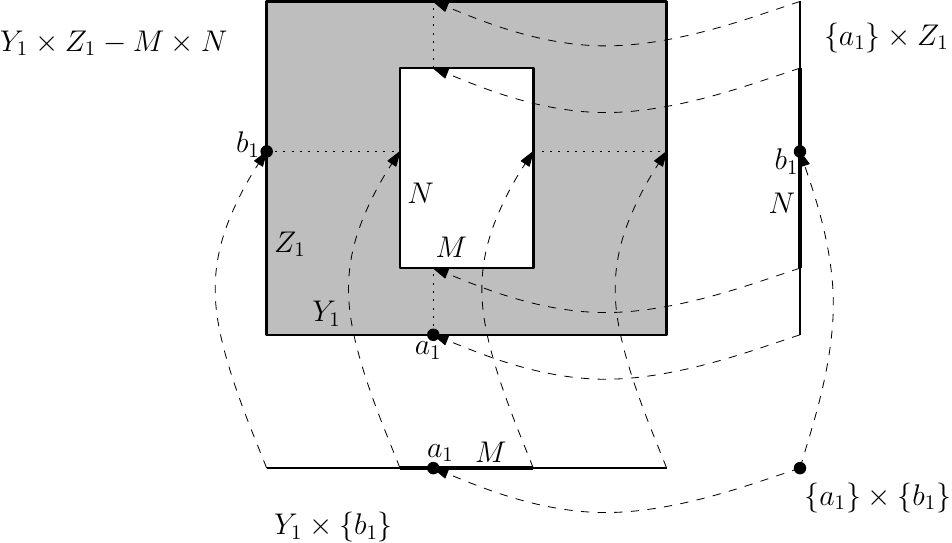}
\caption{Homotopy type of $Y'\times Z' - M \times N$. (Only $a_1$ and $b_1$ showed among $a_i$'s and $b_j$'s.)}
\label{f:EG1}
\end{figure}
Observe that the composition $h_2\circ f$ have the following properties:
\begin{itemize}
	\item $h_2\circ f(Y_1\times \mathbb{R}- M \times I)\subset (Y_1\times Z_1 - M \times N)\cup (Y_1 \times \{b_1\})$;
	\item $h_2\circ f(y,t)=(y,b_1)$, for $t\leqslant c$.
\end{itemize}

Since $Y_1$ is contractible, the restriction of the map  $h_2\circ f$ to the set $(Y_1-M)\times [c,d])\cup (Y_1\times [d,+\infty])$ is homotopic to a constant map (by, basically, ``sliding'' along the path $l$ to
$l(d)$ and
then contracting $h_2\circ f(Y_1\times [d,+\infty])$). 
Since $Y_1,Z_1$ are contractible, therefore (see e.g.\ \cite[Theorem 4.1.8]{Geo-book}), there is a homotopy equivalence
$h_3\colon Y \times Z - M\times N\to \widehat{Y}_1 \vee U$, where $\widehat{Y}_1=\Sigma(Y_1-M)$ is the suspension, and such that $h_3\circ f$ is a homotopy equivalence between $Y_1\times \mathbb{R}- M \times I$ and $\widehat{Y}_1$.
The singular sphere $h_3\circ \varphi \colon S^n \to \widehat{Y}_1$ is then homotopically nontrivial,  hence it is homotopically nontrivial as the map $S^n \to \widehat{Y}_1 \vee U$. It follows that $\varphi \colon
S^n \to Y \times Z - M\times N$ is homotopically nontrivial. This finishes the proof in the case
of $Q$ with infinitely many ends.

\subsection{Case B: $Q$ has two ends.}
\label{s:vircyc}
The idea in the case of virtually cyclic $Q$ is basically the same as before: We construct a
homotopically nontrivial sphere at infinity in $Y_1\times Z$ and then show that this sphere is essential
at infinity in $Y\times Z$. Below we present the details.
\medskip

Recall that we are looking for subcomplexes $M,N$ as described at the end of Subsection~\ref{s:kpi1s},
and that $Z$ is homeomorphic to $\mathbb{R}$ now.
Since $\pi_{n}^{\infty}(Y_1 \times Z)\neq 0$, there exists a subcomplex $M\subset Y_1$ with finite $(n+1)$--skeleton and a compact interval $I\subset \mathbb R$ such that for every compact  $M'\times I' \subset Y_1\times Z$ with finite $(n+1)$--skeleton and containing $M\times I$, there exists a singular sphere $S^{n} \to Y_1 \times Z - M' \times I'$ not homotopically 
trivial within $Y_1 \times Z - M\times I$. In what follows we show that $M$ and $N=I$ are as required.
Let $C\subset Y \times Z$ 
be a subcomplex with finite $(n+1)$--skeleton and containing $M\times N$. Let $M''\subset Y$ and $N'\subset Z$ be subcomplexes with finite $(n+1)$--skeleta such that $C\subset M''\times N'$, and let $\varphi \colon S^n \to
Y_1\times Z - (M''\cap Y_1)\times N'$  be a singular sphere not homotopically trivial within $Y_1 \times Z - M \times N$. We show now that $\varphi$ is not homotopically
trivial within $Y\times Z - M\times N$. Proceeding similarly as in the previous subsection
we obtain a homotopy equivalence $h\colon Y\times Z - M \times N \to (Y_1 \times Z - M \times N)\vee U$, such that $h|_{Y_1 \times Z - M \times N}=\mr{id}_{Y_1 \times Z - M \times N}$. It follows that
$h\circ \varphi \colon S^n \to (Y_1 \times Z - M \times N)\vee U$ is homotopically nontrivial, and hence
$\varphi$ is homotopically nontrivial within $Y\times Z - M\times N$. This finishes the proof in the case of virtually cyclic $Q$ and thus completes the proof of the case (1) of Theorem~\ref{t:2}.

%

\section{Proof of Theorem~\ref{t:0}}
\label{s:pfA}
In this section we prove Theorem~\ref{t:0} and Corollary~\ref{c:0} from the Introduction. 
\medskip

\noindent
\emph{Proof of Theorem~\ref{t:0}:} 
Assume that a finitely presented normal subgroup $K$ is neither finite nor of finite index. Then we have an exact sequence 
$K\rightarrowtail G \twoheadrightarrow Q$ of infinite groups. 
A simply connected SimpHAtic complex $X$ on which $G$ acts geometrically is AHA by \cite[Corollary 3.6]{OS}. Thus $\pi_n^{\infty}(G)=0$, for all $n\geqslant 2$,
by \cite[Theorem 7.5]{OS}.
Since $X$ is also finite-dimensional and contractible (by Whitehead's theorem), by \cite[Theorem 7.7]{OS}, we have that $G$ has type $F_{\infty}$ and $\pi_1^{\infty}(G)\neq 0$. 
Note that a full subcomplex or a covering of a SimpHAtic complex is again SimpHAtic. 
Therefore, by a theorem of Hanlon-Mart{\'{\i}}nez-Pedroza \cite[Theorem 1.1]{HaMP}, the group $K$ acts geometrically on 
a simply connected SimpHAtic complex. By the case (1) of Theorem~\ref{t:2}, it follows that $K$ is virtually free.\hfill $\Box$

\medskip

\medskip

\noindent
\emph{Proof of Corollary~\ref{c:0}:} 
It is clear that a subcomplex of a nonpositively curved (that is, locally CAT(0)) $2$--complex is itself nonpositively curved and two-dimensional, and hence aspherical. Therefore, there exists a subdivision of such complex which is a simplicial complex, hence SimpHAtic.
The SimpHAtic property for weakly systolic complexes with $SD_2^{\ast}$ links follows from \cite[Proposition 8.2]{O-sdn}. In particular, systolic complexes are SimpHAtic (see also \cite{JS1}). Graphical small cancellation complexes are (up to some 
simplicial subdivision) obviously SimpHAtic -- see e.g.\ \cite[Example
 4.4]{OS}.
Therefore, Corollary~\ref{c:0} follows directly from Theorem~\ref{t:0}.
\hfill $\Box$

\medskip
For other examples of SimpHAtic complexes extending the list in Corollary~\ref{c:0} (that is, to which Theorem~\ref{t:0} applies) see
e.g.\ \cite[Section 1.1]{HaMP} and \cite[Section 4]{OS}.

\section{Proofs of Corollaries \ref{c:1}, \ref{c:1a} and \ref{c:3}}
\label{s:fin}
\medskip

\noindent
\emph{Proof of Corollary~\ref{c:1}:} Assume that $K$ is neither finite nor of finite index.
We proceed as in the proof of Theorem~\ref{t:0}.
A finitely generated subgroup of an AHA group is AHA itself \cite[Corollary 3.4]{JS2}, thus $K$ is AHA. 
Furthermore, $K$ has finite virtual cohomological dimension (see e.g.\ \cite[Chapter VIII.11]{Bro}). By \cite[Corollary 7.6]{OS}, $\pi_n^{\infty}(G)=0$, for all $n\geqslant 2$, and by \cite[Theorem 7.7]{OS}, $\pi_1^{\infty}(G)\neq 0$. Therefore $K$ is virtually 
free, by Theorem~\ref{t:2}.\hfill $\Box$

\medskip

\noindent
\emph{Proof of Corollary~\ref{c:1a}:} 
The kernel $K$ is virtually free by Corollary~\ref{c:1}.
Note that virtually cyclic subgroups of hyperbolic groups are quasiconvex -- see \cite[Theorem 8.1.D]{Gro} and \cite[Proposition 3.2]{ABC+}.
Furthermore, it is shown in \cite{ABC+} that a normal quasiconvex subgroup of a hyperbolic group is either finite or of finite index.
Therefore, a virtually cyclic normal subgroup of a hyperbolic group is either finite or of finite index. It follows that $K$ is virtually non-abelian free.

By \cite[Theorem A]{Mos}, the quotient $Q$ is Gromov hyperbolic. Hence, $Q$ acts geometrically on a finite-dimensional contractible complex -- some Rips complex of $Q$. Therefore, by the case (2) of Theorem~\ref{t:2}, the quotient $Q$ is virtually free.\hfill $\Box$

\medskip

\noindent
\emph{Proof of Corollary~\ref{c:3}:} 
The proof is nearly the same as the one of Corollary~\ref{c:1a}. The only point is that instead of the finiteness of the virtual 
cohomological dimension we use the fact that $K$, being SimpHAtic, acts geometrically on a contractible finite-dimensional complex.
\hfill $\Box$


\section{Proofs of Theorems~\ref{t:1a}\&~\ref{t:2a}}
\label{s:t1t2}

\begin{proof}[Proof of Theorem~\ref{t:1a}]
	Let $H<\aut(X)$ be a nonuniform lattice in $X$. 
	With the right normalization of the Haar
	measure $\mu$, due to Serre \cite{Ser1971}, there is a formula for the
	covolume of $H<G$:
	\begin{align*}
		\label{e:1}
		\mu (H\backslash G)=\sum_{v\in Y} \frac{1}{|G_v|}
	\end{align*}
	where $|G_v|$ denotes the order of the stabilizer of $v$, and the sum is taken over  vertices in a 
	fundamental domain $Y$ for the $G$-action on $X$. Since $H$ is nonuniform, the sum on the right is infinite. Since $\mu (H\backslash G)<\infty$ we have that 
	there is no upper bound on the orders of stabilizers $G_v$, hence there is no bound on the orders of finite subgroups in $H$.
	Suppose, by contradiction, that
	$H$ is finitely presented. 
	Since $H$ is discrete, stabilizers of its action on $X$ are finite. Hence, by \cite[Proposition 2.2]{OS}, the group  $H$ is AHA and, by \cite[Theorem C]{OS}, it is of type F$_\infty$.
	By \cite[Proposition]{Kro1993} there is an upper bound on the orders of finite subgroups
	of $H$. This leads to contradiction.	
\end{proof}

Before proving Theorem~\ref{t:2a} we recall some terminology. We follow \cite{HaMP}.
A subcomplex $Y$ of a {combinatorial complex} $X$ is \emph{full} if for every cell $c$ in $X$ if $\partial c\subseteq Y$ then
$c\subseteq Y$. 
A (combinatorial) map $X \to Z$ 
between connected (combinatorial) complexes
is a \emph{tower} if it can be expressed as a composition
of inclusions and covering maps. A \emph{tower lifting} $f'$ of $f$ is a factorization $f=g\circ f'$
where $g$ is a tower. The lifting $f'$  is \emph{trivial} if $g$ is an isomorphism and the lifting is
maximal if the only tower lifting of $f'$ is the trivial one. 
A tower is called an $\mathcal{F}$–tower if it is a composition of covering maps and
inclusions of full subcomplexes.
The following proof is basically the same as the proof of \cite[Theorem 1.1]{HaMP}.

\begin{proof}[Proof of Theorem~\ref{t:2a}]
	Let a group $G$ act properly on a simply connected complex $X$ from $\mathcal{C}$.
	By \cite[Theorem 1.9]{HaMP} there exists a simply connected cocompact $G$-complex
	$X'$ and a $G$-equivariant $\mathcal{F}$-tower $X'\to X$. Clearly $G$ acts properly on $X'$ and
	$X'$ belongs to $\mathcal{C}$. 
\end{proof}

\section*{Appendix: Topological $2$--dimensional quasi-Helly property of systolic complexes}

In this Appendix we present yet another asphericity-like property of systolic complexes. Besides various asymptotic asphericity features 
used extensively through the article -- like e.g.\ AHA or vanishing of higher homotopy pro-groups at infinity \cites{JS2,O-ci,O-ib,OS} -- 
the new property reflects the `two-dimensional-like' nature of systolic complexes, implying the normal subgroup
theorem studied above. The result -- Theorem A below -- is a Helly-type theorem, and we believe that it will be of wide 
use for studying systolic (and, in general, SimpHAtic) complexes and groups.\footnote{In fact, our initial -- not successful, eventually -- approach to the 
normal subgroups results described in the current article, was based on Theorem A.}

Recall the Helly property of the Euclidean plane: If, for a given 
family $\{ A_0,A_1,A_2,A_3 \}$ of convex subsets, all intersections $A_i \cap A_j \cap A_k$ are nonempty,  then 
$A_0 \cap A_1\cap A_2\cap A_3\neq \emptyset$. In \cite{Swie} a quasification of such property is proved for systolic complexes:
It is shown there that the intersection of neighborhoods of corresponding convex subcomplexes has to be non-empty.
Our Theorem A below extends this result: We do not assume that the subcomplexes $A_i$ are convex -- we make assumptions
on homotopy types of them and of their intersections (thus we call our version ``topological''). This is a significant difference: 
The product (with the product metric) of two complexes satisfying the ``convex'' version has the same ``convex'' Helly property, while
for our ``topological'' version the corresponding ``Helly dimension'' may increase by taking products.  
Note also that the topological $2$--dimensional quasi-Helly property presented below is closely related to the 
\emph{$\delta$--thin tetrahedra property} announced in \cite[Introduction]{Els} (there, only geodesic triangles are considered).

Finally, let us note that Theorem A below holds also for a more general class of (SimpHAtic) complexes $X$ satisfying the following 
condition:
\medskip

\noindent
$(\ast)$ Any  simplicial map $S\to X$ from a triangulation $S$ of the $2$--sphere can be extended to a simplicial map $D \to X$ from a triangulation $D$ of the $3$--disc, such that the boundary of $D$ is $S$ and $D$ has no new (internal) vertices, that is $D^{(0)}=S^{(0)}$.

\begin{atheorem}
Let $A_0,A_1,A_2,A_3$ be simply connected subcomplexes of a systolic complex $X$. Assume that for all $i,j,k \in \{ 0,1,2,3 \}$, the intersections $A_i \cap A_j$ are connected, and $A_i\cap A_j \cap A_k \neq \emptyset$.
Then there exists a simplex with vertices $v_0,v_1,v_2,v_3$ such that $v_i \in A_i$ (possibly, some $v_i$ coincide).
\end{atheorem}
\begin{proof}
If $A_0 \cap A_1\cap A_2\cap A_3 \neq \emptyset$ then the assertion is trivial. Therefore further we assume 
that the intersection is empty.

First, we will construct a simplicial $2$--sphere, $\varphi \colon S \to X$. 
For all $l\in \{0,1,2,3\}$, pick a vertex $z_l \in A_i\cap A_j \cap A_k$, where $\{ i,j,k,l \} = \{0,1,2,3 \}$.
For all $l,k$, let $\gamma_{kl}\subset A_i \cap A_j$, where  $\{ i,j,k,l \} = \{0,1,2,3 \}$, be a path connecting $z_k$ and $z_l$.
For $\{ i,j,k,l \} = \{0,1,2,3 \}$, the cycle consisting of $\gamma_{kl}$, $\gamma_{lj}$, and $\gamma_{jk}$ is contained
in $A_i$.
 Therefore, there is a simplicial filling $\varphi_i\colon D_i^2 \to A_i$ of this cycle inside $A_i$, where $D_i^2$ is a triangulation of
the $2$--disc. 
Combining such fillings, we get a simplicial
$2$--sphere $\varphi \colon S \to X$. 

The sphere $S$ has a natural structure of the boundary of the tetrahedron: Vertices are some pre-images $\wt z_i$ of $z_i$'s;
edges are some pre-images $\wt \gamma_{ij}$ of $\gamma_{ij}$'s; triangles are the discs $D_i$'s. However, for our purposes we consider
a ``dual'' structure defined as follows (see Figure~\ref{f:Helly}).
\begin{figure}[h!]
\centering
\includegraphics[width=0.9\textwidth]{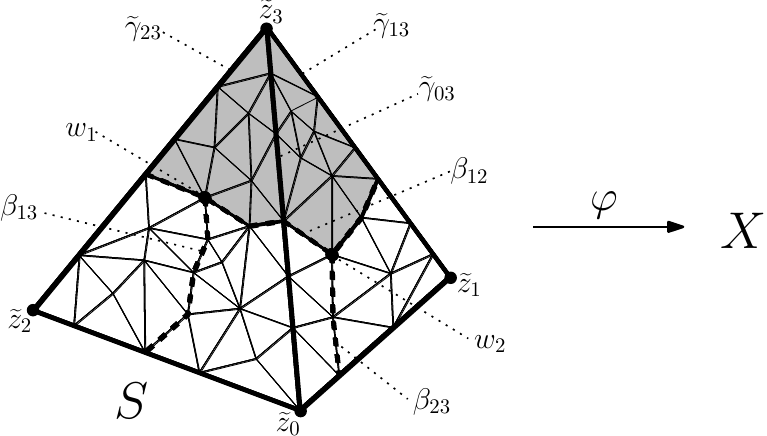}
\caption{The simplicial sphere $\varphi \colon S \to X$. The dual triangle $w_0w_1w_2$ is shaded.}
\label{f:Helly}
\end{figure}
For every $i$, pick a vertex $w_i$ inside the interior of $D_i$ -- subdivide $D_i$ if necessary (and redefine $\varphi$ appropriately). For every $i,j$, find a simple path $\beta_{ij}$ in
$S$ connecting $w_i$ and $w_j$ and contained in $D_i\cup D_j$. Subject to subdividing $S$, we may assume that the intersection
$\beta_{ij} \cap \beta_{kl}$ is either empty or one end-vertex. This defines a new structure of the boundary of the tetrahedron:
Vertices are $w_i$'s; edges are $\beta_{ij}$'s; triangles are connected components of the rest. 
By \cite{JS2} (see \cite[Theorem 2.4]{Els}), there exists a simplicial extension $\Phi \colon D \to X$ of $\varphi$ with the following properties: $D$ is a triangulation of the $3$--disc, with $S$ being a triangulation of the boundary of $D$, and without new (that is, internal) vertices.\footnote{That is, the condition $(\ast)$ above is satisfied.}
Now we define a \emph{coloring} of vertices of $D$, that is a map $c \colon D^{(0)} \to \{ 0,1,2,3 \}$:
\begin{itemize}
\item
$c(w_i):=i$;
\item
$c(v)\in \{ i,j \}$, for $v\in \beta_{ij}$, such that $\varphi(v)\in A_{c(v)}$;
\item
for a vertex $v$ in the triangle $w_iw_jw_k$, we set $c(v)\in \{ i,j,k \}$, such that $\varphi (v) \in A_{c(v)}$.  
\end{itemize}
Observe that the coloring is well-defined (though not unique), and that it is Sperner's coloring of the triangulation of the
 tetrahedron with vertices 
$w_0,w_1,w_2,w_3$. Therefore, by
 Sperner's Lemma \cite{Spe}, there exists a $3$--simplex in $D$ with vertices $v'_0,v'_1,v'_2,v'_3$ 
such that $c(v'_i)=i$, hence $v_i:=\Phi(v'_i) \in A_i$, for all $i\in \{ 0,1,2,3 \}$. Since $\Phi$ is simplicial, it follows that 
$v_0,v_1,v_2,v_3$ are contained in a simplex, and thus the proof is finished.
\end{proof}


\end{document}